\documentclass[a4paper]{amsart}
\usepackage{url}
\usepackage{graphicx}
\usepackage[all]{xy}
\usepackage[mathscr]{eucal}
\usepackage{amsmath,amssymb,amsfonts}
\newtheorem{theorem}{Theorem}[section]
\newtheorem{lemma}[theorem]{Lemma}
\newtheorem{corollary}[theorem]{Corollary}

\newtheorem{proposition}[theorem]{Proposition}
\newtheorem{remark}[theorem]{Remark}

\usepackage{stackengine}
\usepackage{xcolor}
\usepackage[all]{xy}

\title{A generalization of higher rank graphs}

\author{Mark V. Lawson}
\address{Mark V. Lawson, Department of Mathematics
and the
Maxwell Institute for Mathematical Sciences,
Heriot-Watt University,
Riccarton,
Edinburgh EH14 4AS,
UNITED KINGDOM}
\email{m.v.lawson@hw.ac.uk}

\author{Alina Vdovina}
\address{Alina Vdovina,
School of Mathematics and Statistics,
Herschel Building,
Newcastle University,
Newcastle-upon-Tyne NE1 7RU,
UNITED KINGDOM}
\email{alina.vdovina@ncl.ac.uk}

\begin{document}

\begin{abstract}
We introduce what we call `generalized higher rank $k$-graphs' 
as a class of categories equipped with a notion of size.
They extend not only the higher rank $k$-graphs,
but  also the Levi categories introduced by the first author as a categorical setting for graphs of groups.
We prove that examples of generalized higher rank $k$-graphs can be constructed using Zappa-Sz\'ep products of groupoids and higher rank graphs.

\end{abstract}
\maketitle

\section{Introduction}

 Higher rank $k$-graphs are a class of categories  
 important in the theory of $C^{\ast}$-algebras.
They were formally introduced in \cite{KP} by reconceptualizing the work of \cite{RS1999}.
 In this paper, we generalize such categories and show that recent work \cite{ABRW} is a special case leading us to an interesting class of left cancellative categories \cite{JS2014, JS2018}.
What connects this paper and \cite{ABRW} are Zappa-Sz\'ep products of categories.
Specifically, we observed that \cite[Lemma~3.4]{ABRW} has a significant intersection with the axioms (C1)--(C3),
and (SS1)--(SS8) of \cite{Lawson2008b} (also listed in Section~4 below) 
which describe the Zappa-Sz\'ep product of a category by a groupoid.
This means that Zappa-Sz\'ep products of a category by a groupoid can be viewed as
the correct setting for thinking about groupoids acting by partial symmetries on categories.
This was exactly the idea that lay behind \cite{Lawson2008a} where self-similar group actions --- of groups on free monoids --- are viewed as special cases of Zappa-Sz\'ep products of groups on free monoids and are used to construct a class of left cancellative monoids called {\em left Rees monoids}. This idea is explored in \cite{LW2015} and then generalized to categories in \cite{LW2017}; we shall see below that the categories that arise there, called Levi categories, are examples of generalized higher rank $1$-graphs. Levi categories provide a categorical setting for graphs of groups.
In this section, we recall the basic definitions we need to read this paper.
 
 In this paper, categories $C$ are always small and consist only of arrows. 
 The set of identities of $C$  is denoted by $C_{o}$.
 A subcategory of $C$ is said to be {\em wide} if it contains all the identities of $C$.
 If $x \in C$ denote by $\mathbf{d}(x)$ the unique identity such that $x \mathbf{d}(x)$ is defined, and by $\mathbf{r}(x)$ the unique identity such that $\mathbf{r}(x)x$ is defined.
 We call the identity $\mathbf{x}$ the {\em domain} of $x$ and the identity $\mathbf{r}(x)$ the {\em range} of $x$.
 The product $xy$ exists if and only if $\mathbf{d}(x) = \mathbf{r}(y)$;
 we also write $\exists xy$ in this case.
 An arrow $x \in C$ in a category is said to be {\em invertible} if there is an arrow $x^{-1} \in C$ such that
 $x^{-1}x = \mathbf{d}(x)$ and $xx^{-1} = \mathbf{r}(x)$.
 We also say that $x$ is an {\em isomorphism}.
A category in which every element is an isomorphism is called a {\em groupoid}.
If $C$ is a category, we shall denote its set of isomorphisms by $G$.
All identities of a category are isomorphisms.
We shall say that an isomorphism which is not an identity is {\em nontrivial}.
Clearly, $G$ is a wide subcategory of $C$.
More generally, the category $C$ is {\em left cancellative} if $ax = ay$ implies that $x = y$,
and is {\em right cancellative} if $xb = yb$ implies that $x = y$.
A category which is both left and right cancellative is said to be {\em cancellative}.
If $C$ is a category and $a \in C$, define 
$$aC = \{ax \colon x \in C \text{ and } \exists ax \},$$
the {\em principal right ideal generated by $a$},
and define
$$Ca = \{xa \colon x \in C \text{ and } \exists xa \},$$
the {\em principal left ideal generated by $a$}.
The two-sided case $CaC$ has the obvious definition.
Observe that $aC \subseteq \mathbf{r}(a)C$.
Thus every principal right ideal is contained in a principal right ideal
of the form $eC$ where $e$ is an identity.
We say that a principal right ideal $aC$ is {\em maximal} if it is properly contained in 
$\mathbf{r}(a)C$ and is the maximal principal right ideal in $\mathbf{r}(a)C$.
An element $a \in C$ of a category is called an {\em atom} if $a$ is noninvertible and if $a = bc$
 then either $b$ or $c$ is invertible.
 A category $C$ is said to be {\em equidivisible} if $ab = cd$
 implies either that there is an arrow $u$ such that $a = cu$ and $d = ub$
 or there exists an arrow $v$ such that $c = av$ and $b = vd$;
this definition was first stated for monoids \cite{Lallement}.

The set of natural numbers, along with zero, is denoted by $\mathbb{N}$.
We denote by $\mathbb{N}^{k}$ the monoid of all $k$-tuples of elements of $\mathbb{N}$,
and we denote the elements of  $\mathbb{N}^{k}$ by bold lower-case letters such as $\mathbf{m}$.
The additive identity of $\mathbb{N}^{k}$ is denoted by $\mathbf{0}$.
 There is a partial order $\leq$ defined on  $\mathbb{N}^{k}$
 by $\mathbf{m} \leq \mathbf{n}$ if and only if $\mathbf{m} + \mathbf{p} = \mathbf{n}$
 for some $\mathbf{p} \in  \mathbb{N}^{k}$.
 The element of $\mathbb{N}^{k}$ which is zero everywhere except at position $i$
 where it takes the value $1$ is denoted by $\mathbf{e}_{i}$.
 
 A category $C$ is called a {\em higher rank $k$-graph} if there is a functor 
 $\delta \colon C \rightarrow \mathbb{N}^{k}$ satisfying the  {\em unique factorization property} (UFP) as follows:
 if $\delta (a) = \mathbf{m} + \mathbf{n}$, then there exist unique elements $b,c$ such that 
 $a = bc$, $\delta (b) = \mathbf{m}$ and $\delta (c) = \mathbf{n}$.
 It follows that $C$ is a cancellative category with no non-identity invertible elements 
 and $\delta^{-1}(\mathbf{0})$ consists only of identities \cite{KP}.
 Higher rank $k$-graphs can be used to construct groups that include and therefore generalize the classical Thompson groups \cite{LV2020, LSV2021}.
 In the theory of $C^{\ast}$-algebras, the category $C$ is usually assumed to be countable but we do not need this assumption here.

 \section{Categories with size functors}
 
 Let $C$ be a small category.
 By a {\em size functor} on $C$,
  we mean a functor $\lambda \colon C \rightarrow \mathbb{N}^{k}$ such that 
  $\lambda^{-1}(\mathbf{0}) = G$ consists of all and only the invertible elements of $C$. Thus $G$ is the subgroupoid of $C$ consisting of all the invertible elements.
 We begin with some elementary results about categories equipped with size functors. 
 
 \begin{lemma}\label{lem:first} Let $C$ be a category equipped with a size functor 
 $\lambda$ with groupoid of invertible elements $G$.
 \begin{enumerate}
 \item  $aC = bC$ if and only if $aG = bG$.
 \item $Ca = Cb$ if and only if $Ga = Gb$.
 \item $CaC = CbC$ if and only if $GaG = GbG$.
 \end{enumerate}
 \end{lemma}
 \begin{proof} We prove (1) since the proof of (2) follows by symmetry.
 The proof of (3) is then analogous.
 Only one direction needs to be proved.
 Suppose that $aC = bC$.
 Then $a = bx$ and $b = ay$ for some $x,y \in C$.
 Then $a = bx = ayx$.
 Thus $\lambda (a) = \lambda (a) + \lambda (y) + \lambda (x)$.
 It follows that $\mathbf{0} = \lambda (x) + \lambda (y)$.
 Thus $\lambda (x) = \lambda (y) = \mathbf{0}$.
 We deduce that both $x$ and $y$ are invertible.
 The result follows.
 \end{proof}
 
 We adopt semigroup notation and write 
 $a \, \mathscr{R} \, b$ when $aC = bC$
 and 
 $a \, \mathscr{L} \, b$ when $Ca = Cb$.
 We also write $a \, \mathscr{J} \, b$ when $CaC = CbC$.
 
 \begin{lemma}\label{lem:second} Let $C$ be a category equipped with a size functor 
 $\lambda$ with groupoid of invertible elements $G$.
 The element $a$ is invertible if and only if $aC = eC$ for some identity $e$.
 \end{lemma}
 \begin{proof} Suppose that $a$ is invertible, then $aC = aa^{-1}C$.
 Conversely, suppose that $aC = eC$.
 Then by Lemma~\ref{lem:first}, we have that $a = eg = g$ where $g$ is invertible.
 \end{proof}
 
 We now relate atoms to principal right ideals.
 
 \begin{lemma}\label{lem:third}
  Let $C$ be a category equipped with a size functor 
 $\lambda$ with groupoid of invertible elements $G$.
 Then $a \in C$ is an atom if and only if
 $aC$ is maximal.
 \end{lemma}
 \begin{proof} Suppose that $a$ is an atom.
 If $aC \subseteq bC$ then $a = bc$ for some $c \in C$.
 If $b$ is invertible then $bC = bb^{-1}C$.
 If $c$ is invertible then $aC = bC$.
 It follows that $aC$ is maximal.
 Conversely, suppose that $aC$ is maximal and that $a = bc$.
 Then  $aC \subseteq bC$.
 Thus either $aC = bC$ or $bC = eC$ for some identity $e$.
 If the latter then by Lemma~\ref{lem:second}, we have that $b$ is invertible.
 If the former then $a = bg$ for some invertible element $g$.
 It follows that $bc = bg$.
 Thus $\lambda (c) = \mathbf{0}$ and so $c$ is invertible.
 \end{proof}
 
 The fact that the usual order on $\mathbb{N}^{k}$ has the property
 that there are only a finite number of elements below each element
 is used below.
 
 \begin{lemma}\label{lem:fourth} 
 Let $C$ be a category equipped with a size functor 
 $\lambda$ with groupoid of invertible elements $G$.
 Each noninvertible element
 can be written as finite product of atoms.
 \end{lemma}
 \begin{proof} Let $a$ be a noninvertible element.
 If $a$ is an atom then we are done.
 If not, then we can write $a = bc$ where neither $b$ nor $c$ is invertible.
 Observe that $\lambda (a) = \lambda (b) + \lambda (c)$
 and, by assumption, $\lambda (b), \lambda (c) < \lambda (a)$.
 If $b$ and $c$ are atoms then we are done
 otherwise the procedure can be repeated.
 \end{proof}
 
 The groupoid $G$ acts on both sides of the set of atoms.
 
 \begin{lemma}\label{lem:sixth} 
 Let $C$ be a category equipped with a size functor 
 $\lambda$ with groupoid of invertible elements $G$.
 \begin{enumerate}
 \item If $g$ is invertible, $a$ is an atom and $\exists ga$ then $ga$ is an atom.
 \item If $g$ is invertible, $a$ is an atom and $\exists ag$ then $ag$ is an atom.
 \item Let $X$ be a transversal of the set of generators of the maximal principal right ideals.
 Let $g$ be invertible and let $a$ be an atom such that $\exists ga$.
 Then $ga = xh$ for some $x \in X$ and invertible $h$.
 \end{enumerate}
 \end{lemma}
 \begin{proof} (1) First, $ga$ cannot be invertible because this would imply that $a$ was invertible. Suppose that $ga = bc$.
Then $a = (g^{-1}b)c$.
Thus $g^{-1}b$ is invertible in which case $b$ is invertible or $c$ is invertible.
 It follows that $ga$ is an atom.
The proof of (2) is similar to the proof of (1).
(3) We proved in (1) that $ga$ is an atom.
 Thus $gx \, \mathscr{R} \, x$ for some $x \in X$.
 It follows that $gx = xh$ for some $h \in G$
 by Lemma~\ref{lem:first}.
 \end{proof}

We now come to our first main result.

\begin{proposition}\label{prop:fifth} 
Let $C$ be a category equipped with a size functor 
 $\lambda$ with groupoid of invertible elements $G$.
Let $X$ be a transversal of the set of generators of the maximal principal right ideals.
Then $C = \langle X \rangle G$
where $\langle X \rangle$ is the subcategory generated by the elements $X$. 
\end{proposition}
\begin{proof} Let $a \in C$ be any noninvertible element.
Then by Lemma~\ref{lem:fourth}, we may write $a = a_{1} \ldots a_{n}$
where each $a_{i}$ is an atom.
Let $a_{1} \mathscr{R} x_{1}$ where $x_{1} \in X$
by Lemma~\ref{lem:sixth}.
Then $a_{1} = x_{1}h_{1}$ for some $h_{1} \in G$.
Thus $a = x_{1}(h_{1}a_{2})a_{3} \ldots a_{n}$.
We now apply Lemma~\ref{lem:sixth} to the atom $h_{1}a_{2}$.
Repeating this, 
we therefore have that $a = x_{1} \ldots x_{n}h$
 where $x_{i} \in X$ and $h \in G$. 
 \end{proof}

 \section{Generalized higher rank $k$-graphs}

 Proposition~\ref{prop:fifth} gives us some insight into the structure of 
 categories equipped with size functors,
 but to obtain more detailed results, we shall need to assume
 more than the mere existence of a size functor.
 The next definition is the main one of this paper.
 It extends the definition of a higher rank $k$-graph in a natural way by incorporating isomorphisms.\\
 
 \noindent
 {\bf Definition.} 
 Let $C$ be a category equipped with a size functor
 $\lambda$ to $\mathbb{N}^{k}$.
We say that $\lambda$ satisfies the  {\em Weak Factorization Property} (WFP)
if the following holds:
suppose that $\lambda (a) = \mathbf{m} + \mathbf{n}$.
Then there exist $a_{1},a_{2} \in C$ such that $a = a_{1}a_{2}$ where
$\lambda (a_{1}) = \mathbf{m}$ and $\lambda (a_{2}) = \mathbf{n}$ 
such that if $a = b_{1}b_{2}$ where $\lambda (b_{1}) = \mathbf{m}$
and $\lambda (b_{2}) = \mathbf{n}$
then $b_{1} = a_{1}g$ and $b_{2} = g^{-1}a_{2}$ for some invertible element $g$. 
We call such a category a {\em generalized higher rank $k$-graph}.\\

The above definition was motivated by \cite{Lawson2008a, LW2015, LW2017}
combined with the definition of a higher rank $k$-graph.
Observe that higher rank graphs are precisely those generalized higher rank $k$-graphs
which have only the identities as invertible elements (and are countable). 

\begin{lemma}\label{lem:atomic} Let $C$ be a generalized higher rank $k$-graph.
Then the atoms are precisely the elements $a \in C$ 
such that $\lambda (a) = \mathbf{e}_{i}$
for some $i$.
\end{lemma}
\begin{proof} Suppose that $\lambda (a) = \mathbf{e}_{i}$.
We prove that $a$ is an atom.
Let $a = bc$.
Then $\lambda (a) = \lambda (b) + \lambda (c)$.
It follows that either $\lambda (b) = \mathbf{0}$ or $\lambda (c) = \mathbf{0}$
from which we deduce that either $b$ or $c$ is invertible.
Now, suppose that $\lambda (a)$ is zero but not equal to any $\mathbf{e}_{i}$.
Then, for some $i$ we can write $\lambda (a) = \mathbf{e}_{i} + \mathbf{m}$
where $\mathbf{e}_{i}, \mathbf{m} \neq \mathbf{0}$.
Then by the (WFP), we have that $a = bc$ where $\lambda (b) = \mathbf{e}_{i}$
and $\lambda (c) = \mathbf{m}$.
It follows that neither $b$ nor $c$ is invertible and so $a$ cannot be an atom.
\end{proof}

The following is an immediate corollary to Lemma~\ref{lem:atomic}.

\begin{corollary}\label{cor:levi-one} Let $C$ be a generalized higher rank $1$-graph.
Then the atoms are precisely those elements $a \in C$ such that $\lambda (a) = 1$.
\end{corollary}

The following result provides an alternative characterization of generalized higher rank $1$-graphs and shows that our generalization contains other examples apart
from higher rank $k$-graphs.
Recall that a {\em Levi category} \cite{LW2017}
is an equidivisible category equipped with
a  size functor $\lambda \colon C \rightarrow \mathbb{N}$ such that
$\lambda^{-1}(1)$ consists of all and only the atoms of $C$.
We introduced Levi categories in \cite{LW2017} to describe graphs of groups.

\begin{theorem}\label{them:levi-two} \mbox{}
\begin{enumerate}

\item Every Levi category is a  generalized higher rank $1$-graph.

\item Every generalized higher rank $1$-graph is a Levi category.

\end{enumerate}
\end{theorem}
\begin{proof} (1) Let $C$ be a Levi category.
Denote its size functor by $\lambda \colon C \rightarrow \mathbb{N}$. 
Let $a$ be a noninvertible element.
By Lemma~\ref{lem:fourth}, we may write $a = a_{1} \ldots a_{m}$
where the $a_{i}$ are atoms. By assumption, $\lambda (a_{i}) = 1$
and so $\lambda (a) = m$.
Let $m = s + t$.
Then using the above representation of $a$ as a product of atoms, 
we may find elements $x$ and $y$ such that $a = xy$
where $\lambda (x) = s$ and $\lambda (y) = t$.
Specifically, in the case where $s,t \neq 0$ then 
define $x = a_{1} \ldots a_{s}$ and $y = a_{s+1} \ldots a_{m}$.
If $s = 0$ then chose $x = \mathbf{r}(a_{1})$ and $y = a$,
and if $t = 0$ then choose $x = a$ and $y = \mathbf{d}(a_{m})$.
Now suppose that $a = wz$ where $\lambda(w) = s$ and $\lambda (z) =  t$.
Then $a = wz = xy$ where $\lambda (w) = \lambda (x)$ and $\lambda (z) = \lambda (y)$.
We use equidivisibility, and without loss of generality, $w = xu$ and $y = uz$
for some $u \in C$.
Thus $\lambda (u) = 0$ and so $u$ is invertible.
It follows that $w = xu$ and $z = u^{-1}y$.

(2) Let $C$ be a generalized higher rank $1$-graph.
Because of Lemma~\ref{cor:levi-one}, it is enough to prove that $C$ is equidivisible.
Let $ab = cd$.
There are two cases.
We prove the first and the second follows by symmetry.
Suppose that $\lambda (a) \geq \lambda (c)$.
Then $\lambda (c) + m = \lambda (a)$ for some natural number $m$.
By the (WFP), we can write $a = xy$ where $\lambda (x) = \lambda (c)$
and $\lambda (y) = m$.
Thus $x(yb) = cd$.
Observe that $\lambda (x) = \lambda (c)$ and so $\lambda (yb) = \lambda (d)$.
By the (WFP), there is an invertible element $g$ such that
$x = cg$ and $yb = g^{-1}d$.
It follows that $a = xy = c(gy)$ and and $d = (gy)b$.
Put $u = gy$ and we have that $a = cu$ and $d = ub$.
\end{proof}

\section{The structure of a class of generalized higher rank $k$-graphs}

In this section, we shall show how to construct a class of generalized higher rank $k$-graphs.
The following result is a step in this direction.

\begin{lemma}\label{lem:left-cancellation} Let $C$ be a generalized higher rank $k$-graph with groupoid of invertible elements $G$.
Choose and fix a transversal $X$ of the maximal right principal ideals.
Let $\mathscr{X} = \langle X \rangle$.
Then $C$ is left cancellative if and only if 
$ug = u$ implies $g = \mathbf{d}(a)$ where $u \in \mathscr{X}$ and $g \in G$.
\end{lemma}
\begin{proof} 
Only one direction need be proved.
We prove first that if $a \in C$ and $g \in G$ then $a = ag$ implies that $g = \mathbf{d}(a)$.
Let $a \in C$ be noninvertible.
Then $a = uh$ where $u \in \mathscr{X}$ and $h \in G$ by Proposition~\ref{prop:fifth}.
Thus $ag = uhg = uh$.
It follows that $u = uhgh^{-1}$.
By assumption, $\mathbf{d}(u) = hgh^{-1}$ from which we deduce that $g = \mathbf{d}(a)$.
Now, let $ab = ac$.
Then there is an invertible element $g$ such that $a = ag$ and $b = g^{-1}c$.
By assumption, $g = \mathbf{d}(a)$ and so $b = c$, as claimed.
\end{proof}

We shall study generalized higher rank $k$-graphs that satisfy a strengthened form of the condition stated in Lemma~\ref{lem:left-cancellation}.
Let $C$ be a generalized higher rank $k$-graph.
Choose and fix a transversal $X$ of the maximal right principal ideals.
Let $\mathscr{X} = \langle X \rangle$.
We say that $C$ satisfies the {\em $\mathscr{R}$-condition}
if $u,v \in \mathscr{X}$ and $u = vg$ in $C$ implies that $u = v$ and $g$ is an identity.
By Lemma~\ref{lem:left-cancellation}, the category $C$ is left cancellative.
We can now refine Proposition~\ref{prop:fifth}.

\begin{theorem}\label{them:theorem-one} 
Let $C$ be a generalized higher rank $k$-graph satisfying the $\mathscr{R}$-condition,
let $X$ be a transversal of the set of generators of the maximal principal right ideals,
let $\mathscr{X}$ be the wide subcategory of $C$ generated by $X$
and let $G$ be the groupoid of invertible elements of $C$.
Then $\mathscr{X}$ is a higher rank $k$-graph
and every element of $C$ has a unique representation as an element of $\mathscr{X}G$.
\end{theorem}
\begin{proof} Observe that apart from the identities the subcategory $\mathscr{X}$
does not contain any invertible elements because each element non-identity element of $\mathscr{X}$
is a product of atoms and if $a$ is an atom then $\lambda (a) \neq \mathbf{0}$.
It follows that $\mathscr{X}$ is a subcategory with only the identities
as invertible elements.
Let $a \in \mathscr{X}$.
Define $d(a) = \lambda (a)$.
This is clearly a functor.
It remains to show that we have the unique factorization property (UFP).
Let $a \in \mathscr{X}$ be such that $d(a) = \mathbf{m} + \mathbf{n}$.
Then by the (WFP), there are elements $b,c \in C$ such that 
$a = bc$ where $d(b) = \mathbf{m}$ and $d(c) = \mathbf{n}$.
We may write $b = ug$ for some $u \in \mathscr{X}$ and $g \in G$
and we may write $c = vh$ in a similar way.
Thus $a = ugvh$.
Observe that $d(b) = d(u)$ and $d(c) = d(v)$.
We may write $gv = wk$.
Observe that $d(w) = d(v)$.
Thus $a = uwkg$ where $a,u,w \in \mathscr{X}$ and $k,g \in G$.
We now invoke the $\mathscr{R}$-condition,
and deduce that $a = uw$ where $d(u) = \mathbf{m}$ and $d(w) = \mathbf{n}$.
We now prove that $u$ and $w$ are unique with these properties.
Suppose that $a = u_{1}w_{1}$ where $u_{1}, w_{1} \in \mathscr{X}$
and $d(u_{1}) = \mathbf{m}$ and $d(w_{1}) = \mathbf{n}$.
Then by the (WFP), there is an isomorphism $l$ such that
$u_{1} = ul$ and $w_{1} = l^{-1}w$.
By the $\mathscr{R}$-condition, we deduce that $u_{1} = u$ and so $w_{1} = w$.
Thus $\mathscr{X}$ is a higher rank $k$-graph.
It only remains to show that every element has a unique factorization.
Let $ug = vh$ where $u,v \in \mathscr{X}$ and $g,h \in G$.
Then $u = v(hg^{-1})$.
By the $\mathscr{R}$-condition, $hg^{-1} = 1$ and $u = v$;
thus $g = h$, also.
\end{proof}

Our goal now is to show how to construct categories of the above type from
higher rank $k$-graphs and groupoids.
We shall use {\em Zappa-Sz\'ep products} \cite{B2005} or, equivalently, what are termed 
{\em self-similar groupoid actions} in \cite{ABRW}.
The motivation for doing this comes from \cite{Lawson2008a} and the sequence of papers
that followed \cite{Lawson2008b, LW2015, LW2017}.
Let $\mathscr{X}$ be a category and 
let $G$ be a groupoid such that $G_{o} = \mathscr{X}_{o}$.
We suppose that there is a function
$G \times \mathscr{X} \rightarrow \mathscr{X}$ given by $(g,x) \mapsto g \cdot x$,
whenever $\mathbf{d}(g) = \mathbf{r}(x)$, 
and a function
$G \times \mathscr{X} \rightarrow G$ given by $(g, x) \mapsto g|_{x}$,
whenever $\mathbf{d}(g) = \mathbf{r}(x)$,
such that the axioms (C1), (C2), (C3) and (SS1)--(SS8) listed below all hold:
\begin{description}
\item[{\rm (C1)}] $\mathbf{r}(g \cdot x) = \mathbf{r}(g)$.
\item[{\rm (C2)}] $\mathbf{d}(g \cdot x) = \mathbf{r}(g|_{x})$.
\item[{\rm (C3)}] $\mathbf{d}(x) = \mathbf{d}(g|_{x})$.
\end{description} 
This information is summarized by the following diagram:

\centerline{
\xymatrix
{
 & \ar[l]_{g \cdot x}   \\
\ar[u]^{g}&  \ar[u]_{g|_{x}} \ar[l]^{x}
}
}

 \begin{description}
\item[{\rm (SS1)}] $\mathbf{r}(x) \cdot x = x$.
\item[{\rm (SS2)}] $(gh) \cdot x = g \cdot (h \cdot x)$.
\item[{\rm (SS3)}] $g \cdot \mathbf{d}(g) = \mathbf{r}(g)$.
\item[{\rm (SS4)}] $\mathbf{r}(x)|_{x} = \mathbf{d}(x)$.
\item[{\rm (SS5)}] $g|_{\mathbf{d}(g)} = g$.
\item[{\rm (SS6)}] $g|_{xy} = (g|_{x})|_{y}$.
\item[{\rm (SS7)}] $(gh)|_{x} = g|_{h \cdot x}h|_{x}$.
\item[{\rm (SS8)}] $g \cdot (xy) = (g \cdot x)(g_{x} \cdot y)$.
\end{description}

We call this a {\em Zappa-Sz\'ep action} of $G$ on $\mathscr{X}$
or a {\em self-similar action of $G$ on $\mathscr{X}$.}

\begin{remark}\label{rem:home}
{\em Zappa-Sz\'ep products are two-sided semidirect products.
There is a left action of $G$ on $\mathscr{X}$ denoted by $(g,x) \mapsto g \cdot x$
and there is a right action of $\mathscr{X}$ on $G$ denoted by 
$(g,x) \mapsto g|_{\mathbf{d}(g)}$.
The axioms (SS1), (SS2), (SS5) and (SS6) tell us that we are dealing with actions.
Axioms (SS3) and (SS4) tell us how identities act and axioms (SS7) and (SS8)
tell us how the category structures on $G$ and $\mathscr{X}$ interact with the actions.}
\end{remark}

We can use self-similar actions to build categories as we now show.

\begin{proposition}\label{prop:needed} \mbox{}
\begin{enumerate}

\item Let there be a Zappa-Sz\'ep action of the groupoid $G$ on the category $\mathscr{X}$.
Define $\mathscr{X} \bowtie G$ to be the set of all ordered pairs $(x,g)$
where $\mathbf{d}(x) = \mathbf{r}(g)$.
Define 
$$\mathbf{d}(x,g) = (\mathbf{d}(g), \mathbf{d}(g)) 
\text{ and }
\mathbf{r}(x,g) = (\mathbf{r}(x), \mathbf{r}(x)).$$
Define a product by
$$(x,g)(y,h) = (x(g \cdot y), g|_{y}h)$$
when $\mathbf{d}(x,g) = \mathbf{r}(y,h)$.
Then $\mathscr{X} \bowtie G$ is a category 
with set of identities $(e,e)$ where $e \in G_{o} = \mathscr{X}_{o}$.
Define 
$$\iota_{\mathscr{X}} \colon \mathscr{X} \rightarrow \mathscr{X} \bowtie G$$ 
by $x \mapsto (x, \mathbf{d}(x))$. This is an injective functor.
Define 
$$\iota_{G} \colon G \rightarrow \mathscr{X} \bowtie G$$ 
by $g \mapsto (\mathbf{r}(g), g)$. This is an injective functor.
Observe that if $(x,g) \in  \mathscr{X} \bowtie G$ then
$(x,g) = \iota_{\mathscr{X}}(x) \iota_{G}(g)$.

\item Let $C$ be a category with groupoid of isomorphisms $G$ and wide subcategory $\mathscr{X}$ such that $C = \mathscr{X}G$ uniquely.
Then $C$ is isomorphic to $\mathscr{X} \bowtie G$.

\end{enumerate}
\end{proposition}
\begin{proof} (1) The proofs are special cases of general arguments given in \cite{B2005}.
We describe the salient points.
Define 
$\mathbf{d}(x,g) = (\mathbf{d}(g), \mathbf{d}(g))$ 
and 
$\mathbf{r}(x,g) = (\mathbf{r}(x), \mathbf{r}(x))$.
Observe that the product $(x,g)(y,h)$ is defined if and only if $\mathbf{d}(g) = \mathbf{r}(y)$.
It is routine now to check that $\mathscr{X} \bowtie G$ is a category with set of identities
$(e,e)$ where $e \in G_{o} = \mathscr{X}_{o}$.
The product $(x,\mathbf{d}(x))(y, \mathbf{d}(y))$ exists if and only if $\exists xy$
and the product is equal to $(x (\mathbf{d}(x) \cdot y), \mathbf{d}(x)|_{y}\mathbf{d}(y))$
which simplifies to $(xy, \mathbf{d}(y))$ using axioms (SS1) and (SS4).
The product $(\mathbf{r}(g),g)(\mathbf{r}(h), h)$ exists if and only if $\exists gh$
and the product is equal to $(\mathbf{r}(g)(g \cdot \mathbf{r}(h), g|_{\mathbf{r}(h)}h)$
which simplifies to $(\mathbf{r}(g), gh)$ using axioms (SS3) and (SS5).

(2) Let $g \in G$ and $x \in \mathscr{X}$ be such that $gx$ is defined.
Then by assumption, there are unique elements $g \cdot x \in \mathscr{X}$
and $g|_{x} \in G$ such that $gx = (g \cdot x) g|_{x}$.
Checking domains and ranges leads to the verification of axioms (C1), (C2) and (C3)
for a Zappa-Sz\'ep product.
The verification of axioms (SS1)--(SS8) follows from using the following equalitites:
$x = \mathbf{r}(x)x$;
$(gh)x = g(hx)$;
$g = g(g^{-1}g)$;
$g(xy) = (gx)y$. 
It is now routine to check that $C$ is isomorphic to $\mathscr{X} \bowtie G$.
\end{proof}

We now add extra structure to the categories that form the basis of our Zappa-Sz\'ep product.

\begin{lemma}\label{lem:maryland}
Let there be a Zappa-Sz\'ep action of the groupoid $G$ on the category $\mathscr{X}$.
If $\mathscr{X}$ is left cancellative so too is $\mathscr{X} \bowtie G$
\end{lemma}
\begin{proof} Suppose that $(x,g)(y,h) = (x,y)(z,k)$.
Then $x(g \cdot y) = x(g \cdot z)$ and $g|_{y}h = g|_{z}k$.
By left cancellation $g \cdot y = g \cdot z$.
Thus $y = z$  and so $h = k$.
\end{proof}

\begin{lemma}\label{lem:now-now}
Let there be a Zappa-Sz\'ep action of the groupoid $G$ on the category $\mathscr{X}$.
Then
$$g^{-1}|_{x} = (g|_{g^{-1} \cdot x})^{-1}.$$
\end{lemma}
\begin{proof} We calculate from the axioms for a Zapp-Sz\'ep product
$$g|_{g^{-1} \cdot x} g^{-1} |_{x} = \mathbf{r}(g)|_{x}$$
by axiom (SS7).
Thus by axiom (SS4) we get that:
$$g|_{g^{-1} \cdot x} g^{-1} |_{x} = \mathbf{d}(x).$$
The result now follows.
\end{proof}

We now refine
Proposition~\ref{prop:needed}
by incorporating
Lemma~\ref{lem:now-now}
and applying it to the case where
$\mathscr{X}$ is a higher rank $k$-graph.
Denote by $\delta \colon \mathscr{X} \rightarrow \mathbb{N}^{k}$
the usual size functor.
We say that our Zappa-Sz\'ep action is  {\em size-preserving}
if 
$$\delta (g \cdot x) = \delta (x).$$

\begin{proposition}\label{prop:needed-new}
Let there be a size-preserving 
Zappa-Sz\'ep action of the groupoid $G$ on the higher rank $k$-graph $\mathscr{X}$.
Define $\lambda \colon \mathscr{X} \bowtie G \rightarrow \mathbb{N}^{k}$
by $\lambda (x,g) = \delta (x)$.
Then this is a well-defined functor and $\mathscr{X} \bowtie G$
is a generalized higher-rank $k$-graph satisfying the $\mathscr{R}$-condition. 
\end{proposition}
\begin{proof} It is clear that $\lambda$ is a size functor.
Let $(x,g)$ be an element such that $\lambda (x,g) = \mathbf{m} + \mathbf{n}$.
Then by definition $\delta (x) = \mathbf{m} + \mathbf{n}$.
But $\mathscr{X}$ is a higher rank $k$-graph.
Thus there are unique elements $u$ and $v$ such that $x = uv$
where $\delta (u) = \mathbf{m}$ and $\delta (v) = \mathbf{n}$.
Observe that $(x,g) = (u, \mathbf{d}(u))(v,g)$ where 
$\lambda (u, \mathbf{d}(u)) = \mathbf{m}$
and
$\lambda (v,g) = \mathbf{n}$. 
Now suppose that $(x,g) = (w,h)(z,k)$
where $\lambda (w,h) = \mathbf{m}$ vand $\lambda (z,k) = \mathbf{n}$.
Thus $x = w (h \cdot z)$ and $g = h|_{z}k$.
By uniqueness, $w = u$ and $v = h \cdot z$.
Observe that
$(w,h) = (u, \mathbf{d}(u))(\mathbf{r}(h), h)$
and
$(z,k) = (\mathbf{r}(h), h)^{-1}(v,g)$
where we use the fact that
$$(h|_{z})^{-1} = h^{-1}|_{v},$$
a consequence of Lemma~\ref{lem:now-now}.
It is routine to check that the $\mathscr{R}$-condition holds. 
\end{proof}

 We round off this paper with a theorem which pulls together the results of this section.
 
 \begin{theorem}\label{them:one} 
 Let $C$ be a generalized higher rank $k$-graph satisfying the $\mathscr{R}$-condition.
 Denote by $\lambda \colon C \rightarrow \mathbb{N}^{k}$ the size functor.
 Then $C$ is isomorphic to a Zappa-Sz\'ep product  of the form $\mathscr{X} \bowtie G$
 where $\mathscr{X}$ is a higher rank $k$-graph, $G$ is a groupoid
 and the Zappa-Sz\'ep action is size-preserving.
 The functor $\lambda$ is determined by the size functor $\delta$ defined on $\mathscr{X}$.
 \end{theorem}
 \begin{proof} 
 Let $X$ be a transversal of the set of generators 
 of the maximal principal right ideals,
let $\mathscr{X}$ be the wide subcategory of $C$ generated by $X$
and let $G$ be the groupoid of invertible elements of $C$.
Then there is a self-similar action of $G$ on $\mathscr{X}$.
Observe that $\lambda (gx) = \lambda (x) = \lambda (g \cdot x)$.
There is therefore a size-preserving Zappa-Sz\'ep action by Proposition~\ref{prop:needed-new}.
We may therefore form the category $\mathscr{X} \bowtie G$.
Define a function $\theta \colon C \rightarrow \mathscr{X} \bowtie G$
by $\theta (a) = (u,g)$ if $a = ug$ where $u \in \mathscr{X}$ and $g \in G$.
Then $\theta$ is an isomorphism.
Observe that $\lambda (a) = \lambda (ug) = \lambda (u)$.
Define $\delta (u) = \lambda (u)$ when $u \in \mathscr{X}$.
We are done.
\end{proof}


\end{document}